\documentclass[11pt]{amsart}
\usepackage{amssymb}
\usepackage[margin=1in]{geometry}
\usepackage[colorlinks]{hyperref}
\newtheorem{theorem}{Theorem}[section]

\newtheorem{prop}[theorem]{Proposition}
\newtheorem{cor}[theorem]{Corollary}

\theoremstyle{definition}

\theoremstyle{remark}
\newtheorem{remark}[theorem]{Remark}
\newtheorem{notation}[theorem]{Notation}
\numberwithin{equation}{section}
\begin{document}

\newcommand{\spacing}[1]{\renewcommand{\baselinestretch}{#1}\large\normalsize}
\spacing{1.14}

\title[New Randers metrics defined by the other Randers metrics]{New Randers metrics defined by the other Randers metrics}

\author {Azar Fatahi}

\address{Azar Fatahi\\ Department of Pure Mathematics \\ Faculty of  Mathematics and Statistics\\ University of Isfahan\\ Isfahan\\ 81746-73441-Iran.} \email{azi.f3000@yahoo.com}

\author {Masoumeh Hosseini}

\address{Masoumeh Hosseini\\ Department of Pure Mathematics \\ Faculty of  Mathematics and Statistics\\ University of Isfahan\\ Isfahan\\ 81746-73441-Iran.} \email{hoseini\_masomeh@ymail.com}

\author {Hamid Reza Salimi Moghaddam}

\address{Hamid Reza Salimi Moghaddam\\ Department of Pure Mathematics \\ Faculty of  Mathematics and Statistics\\ University of Isfahan\\ Isfahan\\ 81746-73441-Iran.\\ Scopus Author ID: 26534920800 \\ ORCID Id:0000-0001-6112-4259\\} \email{hr.salimi@sci.ui.ac.ir and salimi.moghaddam@gmail.com}


\date{\today}

\begin{abstract}
In this short article, using a left-invariant Randers metric $F$, we define a new left-invariant Randers metric $\tilde{F}$. We show that $F$ is of Berwald (Douglas) type if and only if $\tilde{F}$ is of Berwald (Douglas) type. In the case of Berwaldian metrics, we give the relation between their flag curvatures. Also, we have studied the relations between their base Riemannian metrics. Finally, as examples, the results are studied in the Heisenberg group and almost Abelian Lie groups.
\end{abstract}

\maketitle

\section{\textbf{Introduction}}
$\left(\alpha,\beta\right)$-metrics are a type of Finsler metrics that are known for their simplicity and applications (see \cite{Antonelli-Ingarden-Matsumoto}, \cite{Asanov}, and \cite{Chern-Shen}).
Matsumoto defined these Finsler metrics in \cite{Matsumoto}, but before Matsumoto introduced $(\alpha,\beta)$-metric, in 1941, the first $(\alpha,\beta)$-metric was introduced by G.Randers due
to its application in physics, especially in general relativity, which was called the Randers metric(see \cite{Randers}).\\
Similar to Riemannian metrics, the study of left-invariant Finsler metrics on Lie groups holds a special significance among Finsler metrics.
On the other hand, as we have mentioned, $(\alpha,\beta)$-metrics are important metrics among Finsler metrics.  So the study of left-invariant $(\alpha,\beta)$-metrics
on Lie groups is a very interesting field in Finsler geometry (for example see \cite{Deng-Hosseini-Liu-Salimi} and \cite{salimi}). \\
In this article, we start with a left-invariant Randers metric $F$ and define a new left-invariant Randers metric $\tilde{F}$. First, let us consider the general case on an arbitrary smooth manifold. Suppose that ${\textsf{h}}$ is a Riemannian metric on a smooth manifold $M$ and $X$ is a nowhere zero vector field on $M$ (if there exists such a vector field) with the assumption $\sqrt{{\textsf{h}}(X, X)}<1$, then $F(x,y)=\sqrt{{\textsf{h}}(y,y)}+{\textsf{h}}(X(x),y)$ is a Randers metric on $M$. For any $y\in T_xM\setminus\{0\}$ the fundamental tensor $g_y$ is an inner product on $T_xM$. Hence, we can consider $g_X$ as a Riemannian metric on $M$. According to \cite{Parhizkar-salimi}, for any $v, z\in T_xM$, we have:
\begin{equation}\label{7}
	g_X(v,z)=\Big{(}{\textsf{h}}(v,z)+\frac{1}{\|X(x)\|}{\textsf{h}}(X(x),v){\textsf{h}}(X(x),z)\Big{)}(1+\|X(x)\|),
\end{equation}
where the norm $\|X(x)\|$ is computed by the Riemannian metric ${\textsf{h}}$.\\
Suppose that for any $x\in M$ we have $\sqrt{g_X(X(x),X(x))}<1$. Now, we define a new Randers metric $\tilde{F}$ on $M$ as follows:
\begin{equation}\label{new metric}
	\tilde{F}(x,y)=\sqrt{{g_{X}}(y,y)}+{g_{X}}(X(x),y).
\end{equation}
We mention that one may use $g_y$ and $X$,  $g_X$ and another vector field $Z$, or $g_y$ and $Z$ to define different Randers metrics or other $(\alpha,\beta)$-metrics. However, we have considered $g_X$ and the vector field $X$ for simplicity. Using the same way one may define a sequence of new metrics.\\
If $G$ is a Lie group equipped with a left-invariant Riemannian metric ${\textsf{h}}$, and $X$ is an arbitrary non-zero left-invariant vector field on $G$ with the condition $\|X\|<1$, then easily we see that $F$, which is defined as above, is a left-invariant Randers metric on $G$. Now, using \eqref{7}, we see that the Riemannian metric $g_X$ is a left-invariant Riemannian metric on $G$. On the other hand, if $g_{X}(X,X)<1$, (for example it happens if $\sqrt{{\textsf{h}}(X,X)}<\frac{1}{2}$) then the Finsler metric $\tilde{F}$ defined by \eqref{new metric} is a left-invariant Randers metric on $G$.\\
In this paper, we study the geometry of the left-invariant Riemannian metric $g_X$ and the left-invariant Randers metric $\tilde{F}$ and their relations with the metrics ${\textsf{h}}$ and $F$. More precisely,
in the next section, we focus on the Riemannian metrics $g_X$ and ${\textsf{h}}$. We compute the Levi-Civita connection and the sectional curvature of $g_X$.
Also, we give a necessary and sufficient condition for a geodesic (Killing) vector field of ${\textsf{h}}$ to be a geodesic (Killing) vector field with respect to $g_X$.
In section $3$, the left-invariant Randers metric $\tilde{F}$ which is defined by \eqref{new metric}, is considered, and it is shown that $\tilde{F}$ is of Douglas (Berwald) type if and only if $F$ is of Douglas (Berwald) type.
In the last section, we give some examples of such left-invariant Randers metrics and compute their curvatures on the Lie groups which J. Milnor studied in \cite{Milnor} (almost Abelian Lie groups), and the Heisenberg group (see \cite{salimi2}).

\section{\textbf{The geometry of the Riemannian metric $g_X$}}
In this section, we have considered the left-invariant Riemannian metric $g_X$ defined by \eqref{7}, where ${\textsf{h}}$ is a left-invariant Riemannian metric and $X\in\mathfrak{g}$ is a left-invariant vector field on a Lie group $G$ and $\mathfrak{g}$ denotes the Lie algebra of $G$. We will denote the inner products defined by $g_X$ and ${\textsf{h}}$ on $\mathfrak{g}$ with $g_X$ and $\langle , \rangle$, respectively. We know that there exists a linear mapping $\phi:\mathfrak{g} \to \mathfrak{g}$ such that for any $v,z\in\mathfrak{g}$ we have:
\begin{equation}\label{1}
g_X(v,z)=\langle v,\phi(z)\rangle.
\end{equation}
Now, using equations \eqref{7} and \eqref{1}, easily we can write:
\begin{equation}\label{32}
\phi(z)=(1+\|X\|)z+\frac{1+\|X\|}{\|X\|}\langle X,z\rangle X.
\end{equation}
Since $\phi$ is a symmetric mapping, there exists an orthonormal basis concerning the inner product
$\langle , \rangle$,  of the Lie algebra $\mathfrak{g}$ consisted of eigenvectors of $\phi$.\\
Suppose that $\{X_1,\cdots,X_n\}$ is the above basis consisted of eigenvectors and for $i=1,\cdots,n$, $\lambda_i$ is the eigenvalue corresponded to the eigenvector $X_i$, i.e. $\phi(X_i)=\lambda_iX_i $.

We know that $X=\sum_{j=1}^n\langle X, X_j\rangle X_j$, so for any $i=1\cdots n$ we have:
\begin{equation}\label{new eq}
    \lambda_iX_i=(1+\|X\|)X_i+\frac{1+\|X\|}{\|X\|}\langle X,X_i\rangle \sum_{j=1}^n\langle X,X_j\rangle X_j.
\end{equation}
Easily, for any $1\leq i , j \leq n$, we have:
\begin{equation*}
\lambda_i=(1+\|X\|)(1+\frac{\langle X,X_i\rangle^2}{\|X\|}),
\end{equation*}
and
\begin{equation*}
g_X(X_i,X_j)=\lambda_j \delta_{ij}.
\end{equation*}
On the other hand the equation \eqref{new eq} shows that if for an $i_0$ ($1\leq i_0 \leq n$) we have
$\langle X,X_{i_0}\rangle\neq 0$, then for any $i\neq i_0$ we have $\langle X,X_i\rangle= 0$, which shows that
$X_{i_0}=\frac{X}{\|X\|}$. So we have $\lambda_{i_0}=(1+\|X\|)^2$ and $\lambda_i=1+\|X\|$, where $i\neq i_0$.
\begin{notation}
From now on the set $\{X_1,X_2,\cdots,X_n\}$ denotes the orthonormal basis relative to the inner product $\langle , \rangle$ for $\mathfrak{g}$ consisting of eigenvectors of $\phi$ and $\lambda_1,\lambda_2,\cdots,\lambda_n$ are their  corresponding eigenvalues.
\end{notation}
\begin{prop}\label{29}
Let $(G,{\textsf{h}})$ be a Lie group with a left-invariant Riemannian metric ${\textsf{h}}$. Suppose that $X$ is a left-invariant vector field on $G$ such that $\|X\|<1$, where the norm is computed with respect to ${\textsf{h}}$. Then the Levi-Civita connection of the Riemannian metric $g_X$, defined by the equation \eqref{7} is obtained as follows:
\begin{align*}
\nabla_{X_i}{X_j}&=\sum_{k=1}^{n}\frac{1}{2}\Big{(}\frac{\lambda_i}{\lambda_k}\alpha_{kji}-\frac{\lambda_j}{\lambda_k}\alpha_{ikj}-\alpha_{jik}\Big{)}X_k\\
&=\sum_{k=1}^{n}\frac{1}{2}\Big{(}\frac{\|X\|+\langle X,X_i\rangle^2}{\|X\|+\langle X,X_k\rangle^2}\alpha_{kji}-\frac{\|X\|+\langle X,X_j\rangle^2}{\|X\|+\langle X,X_k\rangle^2}\alpha_{ikj}-\alpha_{jik}\Big{)}X_k,
\end{align*}
where $\alpha_{ijk}$ are structural coefficients with respect to the basis $\{X_1,X_2,\cdots,X_n\}$.
\end{prop}
\begin{proof}
Using the Koszul formula, we have:
\begin{equation}
g_X(\nabla_{X_i}X_j,X_k)=\frac{1}{2}\Big{(}g_X(X_i,[X_k,X_j])-g_X(X_j,[X_i,X_k])-g_X(X_k,[X_j,X_i])\Big{)}.
\end{equation}
It follows from $[X_i,X_j]=\sum_{k=1}^{n}\alpha_{ijk}X_k$ and \eqref{1} that
\begin{equation}\label{2}
g_X(\nabla_{X_i}X_j,X_k)=\frac{1}{2}(\lambda_i\alpha_{kji}-\lambda_j\alpha_{ikj}-\lambda_k\alpha_{jik}).
\end{equation}
Now, the fact that,
\begin{equation}\label{3}
\nabla_{X_i}{X_j}=\sum_{k=1}^{n}\langle\nabla_{X_i}{X_j},X_k\rangle X_k=\sum_{k=1}^{n}\frac{1}{\lambda_k}g_X(\nabla_{X_i}X_j,X_k)X_k,
\end{equation}
together with a direct computation, completes the proof.
\end{proof}

\begin{prop}\label{31}
Under the same assumptions as in Proposition \ref{29}, if $K$ is the sectional curvature of Riemannian metric $g_X$, then the formula for $K$ is as follows:
\begin{equation}\label{6}
K(X_i,X_j)=\frac{1}{4} \sum_{l=1}^{n}\Big{(}-\frac{4}{\lambda_l}\alpha_{ljj}\alpha_{lii}-\frac{2}{\lambda_l}\alpha_{lji}\alpha_{ilj}+\frac{2}{\lambda_j}\alpha_{ijl}\alpha_{jli}+\frac{2}{\lambda_i}\alpha_{ilj}\alpha_{jil}-\frac{3\lambda_l}{\lambda_i \lambda_j}\alpha_{ijl}^2+\frac{\lambda_j}{\lambda_i \lambda_l}\alpha_{ilj}^2+\frac{\lambda_i}{\lambda_j \lambda_l}\alpha_{lji}^2\Big{)}.
\end{equation}
\end{prop}
\begin{proof}
Easily, the sectional curvature of Riemannian metric $g_X$ is as follows:
\begin{equation}\label{5}
K(X_i,X_j)= \frac{1}{\lambda_j}\langle R(X_i,X_j)X_j,X_i\rangle.
\end{equation}
On the other hand, for the curvature tensor of the Riemannian metric we have:
\begin{equation}
	R(X_i,X_j)X_k=\nabla_{X_i}\nabla_{X_j}X_k-\nabla_{X_j}\nabla_{X_i}X_k-\nabla_{[X_i,X_j]}X_k.
\end{equation}
Now, a direct computation gives:
\begin{equation}	 \nabla_{X_i}\nabla_{X_j}X_k=\frac{1}{4}\sum_{l,h=1}^{n}(\frac{\lambda_j}{\lambda_l}\alpha_{lkj}-\frac{\lambda_k}{\lambda_l}\alpha_{jlk}-\alpha_{kjl})(\frac{\lambda_i}{\lambda_h}\alpha_{hli}-\frac{\lambda_l}{\lambda_h}\alpha_{ihl}-\alpha_{lih})X_h.
\end{equation}

\begin{equation}
\nabla_{X_j}\nabla_{X_i}X_k=\frac{1}{4}\sum_{l,h=1}^{n}(\frac{\lambda_i}{\lambda_l}\alpha_{lki}-\frac{\lambda_k}{\lambda_l}\alpha_{ilk}-\alpha_{kil})(\frac{\lambda_j}{\lambda_h}\alpha_{hlj}-\frac{\lambda_l}{\lambda_h}\alpha_{jhl}-\alpha_{ljh})X_h.
\end{equation}

Also due to the equation $[X_i,X_j]=\sum_{k=1}^{n}\alpha_{ijk}X_k$, we have:
\begin{equation}
	\nabla_{[X_i,X_j]}X_k=\frac{1}{2}\sum_{l,h=1}^{n} \alpha_{ijl}(\frac{\lambda_l}{\lambda_h}\alpha_{hkl}-\frac{\lambda_k}{\lambda_h}\alpha_{lhk}-\alpha_{klh})X_h.
\end{equation}
Hence we obtain
\begin{align}\label{4}
R(X_i,X_j)X_k&=\frac{1}{4}\sum_{l,h=1}^{n}\{(\frac{\lambda_j}{\lambda_l}\alpha_{lkj}-\frac{\lambda_k}{\lambda_l}\alpha_{jlk}-\alpha_{kjl})(\frac{\lambda_i}{\lambda_h}\alpha_{hli}-\frac{\lambda_l}{\lambda_h}\alpha_{ihl}-\alpha_{lih}) \nonumber\\
	 &-(\frac{\lambda_i}{\lambda_l}\alpha_{lki}-\frac{\lambda_k}{\lambda_l}\alpha_{ilk}-\alpha_{kil})(\frac{\lambda_j}{\lambda_h}\alpha_{hlj}-\frac{\lambda_l}{\lambda_h}\alpha_{jhl}-\alpha_{ljh})\nonumber\\
	 &-2\alpha_{ijl}(\frac{\lambda_l}{\lambda_h}\alpha_{hkl}-\frac{\lambda_k}{\lambda_h}\alpha_{lhk}-\alpha_{klh})\}X_h.
\end{align}
So the equations \eqref{5} and \eqref{4}, complete the proof.
\end{proof}
\begin{prop}
Under the same assumptions as in Proposition \ref{29} we have:
\begin{enumerate}
  \item If $v$ is a geodesic vector field with respect to the Riemannian metric ${\textsf{h}}$, then $v$ is a geodesic vector field with respect to the Riemannian metric $g_X$, if and only if
	\begin{equation}\label{8}
		ad_{v}^*X=0\quad or\quad \langle X,v\rangle=0.
	\end{equation}
  \item If $v$ is a Killing vector field with respect to the Riemannian metric ${\textsf{h}}$, then $v$ is a Killing vector field with respect to the Riemannian metric $g_X$ if and only if
	\begin{equation}\label{9}
		ad_{v}^*X=0,
	\end{equation}
\end{enumerate}
where $ad^*_v$ denotes the transpose of $ad_v$ with respect to $\langle , \rangle$.
\end{prop}
\begin{proof}
For (1), we recall that, using \cite{Latifi}, $v$ is a geodesic vector field with respect to $g_X$, if and only if:
	\begin{equation}
		g_X(v,[v,y])=0 \quad \forall y\in \mathfrak{g}.
	\end{equation}
Now, according to \eqref{7} and because $v$ is the geodesic vector field concerning the Riemannian metric $\langle , \rangle$, we have:
	 \begin{equation*}
	 	\langle X,v\rangle \langle X,[v,y]\rangle=0,
	 \end{equation*}
which completes the proof.\\
For (2), we mention that $v$ is a Killing vector field concerning $g_X$ if and only if:
\begin{equation}
	g_X([v,y],z)+g_X(y,[v,z])=0 \quad \forall y,z \in \mathfrak{g}.
\end{equation}
Now, the equation \eqref{7}, together with the fact that $v$ is the Killing vector field for the Riemannian metric ${\textsf{h}}$, shows that the above equation is equivalent to the following equation:
\begin{equation}\label{10}
	\langle X,z\rangle \langle X,[v,y]\rangle+\langle X,y\rangle \langle X,[v,z] \rangle=0 \qquad \forall y,z.
\end{equation}
Now, if we put $z=y$, then the result is concluded.
\end{proof}


\section{\textbf{The geometry of the Randers metric defined by $g_X$ and $X$}}
In this section, using a Randers metric $F$, we define a new Randers metric $\tilde{F}$ and study some of their geometric relations.
Suppose that $(M, F)$ is a Randers space, where $F$ is defined by a Riemannian metric ${\textsf{h}}$ and a vector field $X$. Let $g_y$ be the fundamental tensor of the Randers metric $F$. As mentioned in the previous section, $g_X$ is a Riemannian metric on $M$. Suppose that $\sqrt{g_X(X, X)}<1$ (we recall that, for example, this condition holds if $\sqrt{{\textsf{h}}(X, X)}<\frac{1}{2}$). We can easily see the equation \ref{new metric} defines a Randers metric on $M$. A direct computation shows that
\begin{equation}\label{New metric simple formula}	 \tilde{F}(x,y)=(\|y\|^2+\frac{{\textsf{h}}(X(x),y)^2}{\|X\|})^{\frac{1}{2}}(1+\|X\|)^\frac{1}{2}
+{\textsf{h}}(X(x),y)(1+\|X\|)^2,
\end{equation}
where the norm $\| \|$ is computed by ${\textsf{h}}$.\\
We note that, during this article, when we use the Randers metric $\tilde{F}$, we have considered
$\sqrt{g_X(X, X)}<1$.

\begin{prop}\label{11}
Let $G$ be a Lie group with a left-invariant Randers metric $F$ arising from a left-invariant Riemannian metric ${\textsf{h}}$ and a left-invariant vector field $X$.
\begin{enumerate}
  \item $F$ is of Douglas type if and only if $\tilde{F}$ is of Douglas type.
  \item $F$ is of Berwald type if and only if $\tilde{F}$ is of Berwald type.
\end{enumerate}
\end{prop}
\begin{proof}
For (1), we know that the Randers metric $F$ is of Douglas type if and only if  $\langle X,[\mathfrak{g},\mathfrak{g}] \rangle=0$, where similar to the previous section $\langle, \rangle$  is the inner product induced by the Riemannian metric ${\textsf{h}}$ on $\mathfrak{g}$. On the other hand, the equations \eqref{1} and \eqref{32} together with a little computations show that $\langle X, [\mathfrak{g},\mathfrak{g}]\rangle=0$ if and only if $g_X(X,[\mathfrak{g},\mathfrak{g}])=0$, so the proof is completed.\\
For the part (2), by Proposition 4.1 in \cite{Deng-Hosseini-Liu-Salimi}, the Randers metric $F$ is of Berwald type if and only if
\begin{align}
	&\langle X,[y,z] \rangle=0,\\
	&\langle [y,X],z \rangle+\langle [z,X],y \rangle=0.
\end{align}
Consider the basis defined in Proposition \ref{29}. Suppose that $F$ is of Berwald type. So $F$ (and also $\tilde{F}$) is of Douglas type and for any $1\leq i, j \leq n$ we have $\langle [X_i,X],X_j \rangle+\langle [X_j,X],X_i \rangle=0$, which proves that $g_X([X_i,X],X_j)+g_X([X_j,X],X_i)=0$.\\
Conversely, let $\tilde{F}$ be of Berwald type. Therefore, $\tilde{F}$ (and also $F$) is of Douglas type and
\begin{equation}
	g_X([X_i,X],X_j)+g_X([X_j,X],X_i)=(1+\|X\|)(\langle X_j,[X_i,X] \rangle+\langle X_i,[X_j,X] \rangle)=0.
\end{equation}
The last equation completes the proof.
\end{proof}
If $F$ (or equivalently $\tilde{F}$) is of Douglas type, then the formulas of Levi-Civita connections of ${\textsf{h}}$ and $g_X$ are simple. In this case, we have the following proposition.
\begin{prop}\label{proposition nabla}
Suppose that $\nabla$ and $\bar{\nabla}$ denote the Levi-Civita connections of $g_X$ and ${\textsf{h}}$, respectively. Under the same assumptions as in Proposition (\ref{11}), if $F$ is of Douglas type then for $\nabla$ and $\bar{\nabla}$ we have:
\begin{eqnarray}\label{22}
  \nabla_{X_i}X_j&=&\sum_{k=1}^{n}\frac{1+\|X\|}{2\lambda_k}(\alpha_{kji}-\alpha_{ikj}-\alpha_{jik})X_k \nonumber \\
  &=& \bar{\nabla}_{X_i}X_j + \frac{1}{2\|X\|(1+\|X\|)}(\langle [X_i,X],X_j \rangle + \langle [X_j,X],X_i \rangle ) X.
\end{eqnarray}
\end{prop}
\begin{proof}
The Koszul formula
\begin{equation}\label{12}
	g_X(\nabla_{X_i}X_j,X_k)=\frac{1}{2}(g_X(X_i,[X_k,X_j])-g_X(X_j,[X_i,X_k])-g_X(X_k,[X_j,X_i])),
\end{equation}
together with the equations \eqref{1} and \eqref{12} show that
\begin{equation}\label{13}
		g_X(\nabla_{X_i}X_j,X_k)=\frac{1+\|X\|}{2}(\alpha_{kji}-\alpha_{ikj}-\alpha_{jik}).
\end{equation}
As well as
\begin{equation}\label{14}
		g_X(\nabla_{X_i}X_j,X_k)=(1+\|X\|)\langle\bar{\nabla}_{X_i}X_j,X_k\rangle.
\end{equation}
Now the equation
\begin{equation}\label{15}
	\nabla_{X_i}X_j=\sum_{k=1}^{n}\langle\nabla_{X_i}X_j,X_k\rangle X_k
\end{equation}
completes the proof of the first equation.\\
For the second equation, we recall that the Randers metric $\tilde{F}$ is of Douglas type if and only if $g_X(X,[y,z])=0$. Now according to \eqref{14} we have
	\begin{equation}\label{16}
		\phi(\nabla_{X_i}X_j)=(1+\|X\|)\bar{\nabla}_{X_i}X_j.
	\end{equation}
On the other hand, it is known that (see \cite{Arvanitoyeorgos}):
	\begin{equation}\label{18}
		\nabla_{X_i}X_j=\frac{1}{2}([X_i,X_j]-ad^*_{X_i}X_j-ad^*_{X_j}X_i),
	\end{equation}
where the transpose is computed with respect to $g_X$.
Due to the \eqref{7} we have
	\begin{equation}\label{19}
		g_X(X,ad^*_{X_i}X_j)=(1+\|X\|)^2\langle X,ad^*_{X_i}X_j\rangle,
	\end{equation}
and
	\begin{equation}\label{20}
		\langle\phi([X_i,X]),X_j\rangle=(1+\|X\|)\langle [X_i,X],X_j\rangle.
	\end{equation}
Now according to \eqref{18}, \eqref{19}) and \eqref{20} we have:
	\begin{equation}\label{21}
		\langle X,\nabla_{X_i}X_j\rangle=-\frac{1}{2(1+\|X\|)}(\langle [X_i,X],X_j\rangle +\langle [X_j,X],X_i\rangle.
	\end{equation}
On the other hand, we know
	\begin{equation}\label{17}
		\phi(\nabla_{X_i}X_j)=(1+\|X\|)\nabla_{X_i}X_j+\frac{1+\|X\|}{\|X\|}\langle X,\nabla_{X_i}X_j\rangle X.
	\end{equation}
	Now substituting the relations \eqref{21} and \eqref{16} in \eqref{17}, completes the proof.
\end{proof}

\begin{cor}
In the previous proposition, if $X$ belongs to the center of Lie algebra of $G$ or $X$ is a Killing vector field of the Riemannian metric ${\textsf{h}}$, then the Levi-Civita connection of ${\textsf{h}}$ and the Levi-Civita connection of $g_X$ coincide. In this case, the sectional curvature of the Riemannian metric $g_X$ is as follows:
\begin{equation}
	K(X_i,X_j)=\frac{1}{1+\|X\|}\bar{K}(X_i,X_j),
\end{equation}
where $K$ and $\bar{K}$ denote the sectional curvature of $g_X$ and ${\textsf{h}}$, respectively.
\end{cor}
\begin{proof}
A direct computation together with \eqref{22}, concludes the proof.
\end{proof}
\begin{cor} \label{23}
With the assumptions of Proposition \ref{11}, if $F$ (or equivalently $\tilde{F}$) is of Berwald type, then the Levi-Civita connection of ${\textsf{h}}$ and the Levi-Civita connection of $g_X$ coincide.
\end{cor}
\begin{prop}\label{25}
Under the same assumptions as in Corollary \ref{23}, the sectional curvature of the Riemannian metric $g_X$ is given by
	\begin{equation}
		K(X_i,X_j)=\frac{1}{1+\|X\|}\bar{K}(X_i,X_j).
	\end{equation}
\end{prop}
\begin{proof}
Using Corollary \eqref{23}, we have $R(X_i,X_j)X_j=\bar{R}(X_i,X_j)X_j$, where $\bar{R}$ denotes the curvature tensor of the Riemannian metric ${\textsf{h}}$. Now it is sufficient to recall that, by equation \eqref{1}, we have:
	\begin{equation}\label{24}
		K(X_i,X_j)= \frac{1}{\lambda_j}\langle R(X_i,X_j)X_j,X_i\rangle.
	\end{equation}
\end{proof}
Here we mention that $K(X_{i_0},X_j)=\bar{K}(X_{i_0},X_j)=0$.
\begin{prop}
With the same assumptions as in Proposition \ref{proposition nabla}, the sectional curvature of the Riemannian metric $g_X$ is given by
 \begin{eqnarray*}
   K(X_i,X_j)&=&\frac{1}{1+\|X\|}\Big\{\bar{K}(X_i,X_j)+\frac{1}{\|X\|(1+\|X\|)}\big(\langle [X_i,X],X_i\rangle\langle[X_j,X],X_j\rangle \\
   && -\frac{1}{4}(\langle[X_i,X],X_j\rangle+\langle[X_j,X],X_i\rangle)^2\big)\Big\}.
 \end{eqnarray*}
\end{prop}
\begin{proof}
We know the curvature tensor of the Riemannian metric $g_X$ is as follows:
\begin{equation}
		R(X_i,X_j)X_k=\nabla_{X_i}\nabla_{X_j}X_k-\nabla_{X_j}\nabla_{X_i}X_k-\nabla_{[X_i,X_j]}X_k.
\end{equation}
Now, using Proposition \ref{proposition nabla} we have:
\begin{align*}
		R(X_i,X_j)X_k&= \bar{R}(X_i,X_j)X_k+
		\frac{1}{2\|X\|(1+\|X\|)}\Big\{ \Big(\langle[X_i,X],\bar{\nabla}_{X_j}X_k\rangle\\
		&+\langle[\bar{\nabla}_{X_j}X_k,X],X_i\rangle- \langle[X_j,X],\bar{\nabla}_{X_i}X_k\rangle - \langle[\bar{\nabla}_{X_i}X_k,X],X_j\rangle-\langle[[X_i,X_j],X],X_k\rangle\\
		&-\langle[X_k,X],[X_i,X_j]\rangle \Big)X +\Big(\langle[X_j,X],X_k\rangle+\langle[X_k,X],X_j\rangle\Big)\bar{\nabla}_{X_i}X\\
		&-\Big(\langle[X_i,X],X_k\rangle+\langle[X_k,X],X_i\rangle\Big)\bar{\nabla}_{X_j}X\Big\},
\end{align*}
where $\bar{R}$ denotes the curvature tensor of the Riemannian metric ${\textsf{h}}$. We also know that the sectional curvature of the Riemannian metric $g_X$ is as follows:
	\begin{equation}
		K(X_i,X_j)=\frac{1}{\lambda_j}\langle R(X_i,X_j)X_j,X_i\rangle
	\end{equation}
Then with a direct computation, we have:
\begin{align*}
	 K(X_i,X_j)=\frac{1}{\lambda_j}\Big\{\bar{K}(X_i,X_j)-&\frac{1}{\|X\|(1+\|X\|)}\Big(\langle\bar{R}(X_i,X_j)X_j,X\rangle \langle X,X_i\rangle - \langle[X_j,X],X_j\rangle \langle[X_i,X],X_i\rangle\Big)\\ &-\frac{1}{4\|X\|(1+\|X\|)} \Big(\langle[X_i,X],X_j\rangle+\langle[X_j,X],X_i\rangle\Big)^2 \Big\}
\end{align*}
\end{proof}


In the following, we compute the flag curvature of the Randers metric $\tilde{F}$, in terms of the flag curvature of the Randers metric $F$, where they are of Berwald type.
\begin{prop}
Assume that $G$ is a Lie group with a left-invariant Berwaldian Randers metric $F$ arising from a left-invariant Riemannian metric ${\textsf{h}}$ and a left-invariant vector field $X$. Then the flag curvature $K^{\tilde{F}}(X_j,P)$ of $\tilde{F}$ is given by
\begin{equation}
	K^{\tilde{F}}(X_j,P)=\frac{(1+\langle X,X_j\rangle)^2}{\lambda_j(1+\lambda_{j}^\frac{1}{2}\langle X,X_j\rangle)^2}K^F(X_j,P),
\end{equation}
where $P=span\{X_i,X_j\}$ and $K^F(X_i,X_j)$ denotes the flag curvature of $F$.
\end{prop}
\begin{proof}
	By Proposition \ref{11}, $\tilde{F}$ is of Berwald type. According to the Proposition 3.1 in \cite{Deng-Hosseini-Liu-Salimi} , the flag curvatures $K^F$ and $K^{\tilde{F}}$ of $F$ and $\tilde{F}$, respectively, are as follows:
	\begin{equation}\label{28}
		K^{\tilde{F}}(X_j,P)=\frac{g_X(X_j,X_j)}{\tilde{F}^2(X_j)}K(X_i,X_j),
	\end{equation}
	\begin{equation}\label{27}
		K^F(X_j,P)=\frac{{\textsf{h}}(X_j,X_j)}{F^2(X_j)}\bar{K}(X_i,X_j),
	\end{equation}
where, as above, $K$ is the sectional curvature of the Riemannian metric $g_X$ and $\bar{K}$ is the sectional curvature of the Riemannian metric ${\textsf{h}}$.\\
Using the Proposition \ref{25}, we have
	\begin{equation}\label{26}
		K(X_j,X_j)=\frac{1}{\lambda_{j}}\bar{K}(X_j,X_j).
	\end{equation}
On the other hand, we have
	\begin{equation*}
		{\textsf{h}}(X_j,X_j)=1,\qquad F^2(X_j)=(1+\langle X,X_j\rangle)^2,
	\end{equation*}
	and
	\begin{equation*}
		 g_X(X_j,X_j)=\lambda_{j},\qquad \tilde{F}^2(X_j)=\lambda_{j}(1+\lambda_{j}^\frac{1}{2}\langle X,X_j\rangle)^2.
	\end{equation*}
Now, by \eqref{26}, \eqref{27} and \eqref{28}, we conclude that \eqref{27} holds.	
\end{proof}


\section{\textbf{Examples}}
\subsection{The Heisenberg group $H_3$}
Let $H_3$ be the simply connected 3-dimensional Heisenberg Lie group, and $\langle, \rangle$ be the inner product on its Lie algebra ${\mathfrak{h}}_3$, induced by the left-invariant Riemannian metric ${\textsf{h}}$ on $H_3$ introduced in \cite{salimi2}. Suppose that  $\{x,y,z\}$ is an orthogonal basis of ${\mathfrak{h}}_3$ and
$$\left(\begin{array}{ccc}
	\lambda &0&0\\
	0&\lambda&0\\
	0&0&1
\end{array}\right)$$
is the matrix of the inner product $\langle , \rangle$ with respect to this basis, where $\lambda>0$. \\
Suppose that $X$ is a left-invariant vector field belonging to the center of the Lie algebra $\mathfrak{h}_3$. So $X$ is a scalar multiple of $z$. Let $X=cz$ and  $\|X\|=|c|<\frac{1}{2}$. We can see that the linear mapping  $\phi:\mathfrak{h}_3\to \mathfrak{h}_3$ is as follows:
\begin{equation}\label{30}
\phi(v)=(1+\vert c \vert)(v+\vert c \vert\langle z,v \rangle z).
\end{equation}
Easily the matrix representation of $\phi$ with respect to the basis $\{x,y,z\}$ is of the form
$$\left(\begin{array}{ccc}
	1+\vert c \vert &0&0\\
	0&1+\vert c \vert&0\\
	0&0&(1+\vert c \vert)^2
\end{array}\right).$$
Therefore, its eigenvalues are $\lambda_1=1+\vert c \vert$, $\lambda_2=1+\vert c \vert$, $\lambda_3=(1+\vert c \vert)^2$ with the corresponding eigenvectors $X_1=\frac{x}{\sqrt{\lambda}}$, $X_2=\frac{y}{\sqrt{\lambda}}$ and $X_3=z$, which construct an orthonormal basis for $\mathfrak{h}_3$. So we have
\begin{align}
	[X_1,X_2]&=\frac{1}{\lambda}z\\
	[X_1,X_3]&=[X_2,X_3]=0.
\end{align}
Therefore the non-zero structure constants are $\alpha_{123}=-\alpha_{213}=\frac{1}{\lambda}$.\\
Now according to Proposition \ref{29} the Levi-Civita connection of the Riemannian metric $g_X$ is as follows:
\begin{align*}
\nabla_x y&=\frac{1}{2}z,\quad \nabla_x z=-\frac{1+\vert c \vert}{2\lambda}y,\\
 \nabla_y z&=\frac{1+\vert c \vert}{2\lambda}x, \quad \nabla_y x=-\frac{1}{2}z, \\
 \nabla_z x&=-\frac{1+\vert c \vert}{2\lambda}y, \quad \nabla_z y=\frac{1+\vert c \vert}{2\lambda}x,\\
 \nabla_x x&=\nabla_y y = \nabla_z z=0.
\end{align*}
Due to the \eqref{7} and \eqref{30} we have
\begin{align*}
	g_{X}(x,x)&=g_{X}(y,y)=\lambda(1+\vert c \vert),\quad g_{X}(z,z)=(1+\vert c \vert)^2,\\
	g_{X}(x,y)&=g_{X}(x,z)=g_{X}(z,y)=0.
\end{align*}
Therefore, according to Proposition \ref{31} the sectional curvature of the Riemannian metric $g_X$ is as follow
\begin{align*}
K(x,y)&=\frac{-3}{4\lambda^2},\\
K(x,z&)=K(y,z)=\frac{1}{4\lambda^2}.
\end{align*}
\subsection{The almost Abelian Lie groups}
A non-Abelian Lie group $G$ is called almost Abelian if it has
a codimension one Abelian subgroup. Suppose that $G$ is an almost Abelian Lie group with Lie algebra $\mathfrak{g}$. Let $\langle, \rangle$ be a left-invariant Riemannian metric on $G$. Then there exist $b\in\mathfrak{g}$ and a commutative ideal $\mathfrak{u}$ of codimention one such that $[b,z]=z$, for any $z\in\mathfrak{u}$ (see \cite{Milnor} and \cite{Deng-Hosseini-Liu-Salimi}). Suppose that $\{u_1=b,u_2,u_3,\cdots,u_n\}$ is an orthonormal basis for $\mathfrak{g}$ with respect to $\langle, \rangle$,
where $\{u_2,u_3,\cdots,u_n\}$ is a basis for $\mathfrak{u}$. In this case we have, $[b,u_i]=u_i,\quad i=2,\cdots,n$ and $[u_i,u_j]=0,\quad i,j\neq1$. Therefore, the non-zero structural constants are $\alpha_{1ii}=-\alpha_{i1i}=1,\quad i=2,3,\cdots,n$. Let $X=\xi b$, where in $\|X\|=\xi<\frac{1}{2}$. Then the linear mapping  $\phi:\mathfrak{g}\to \mathfrak{g}$ is as follows:
\begin{equation}
	\phi(v)=(1+\xi)v+(1+\xi)\xi\langle b,v\rangle b.
\end{equation}
Then, the matrix representation of $\phi$ with respect to the ordered basis $\{u_1=b,u_2,u_3,\cdots,u_n\}$ is as follows
$$\left(\begin{array}{cccccc}
	(1+\xi)^2 &0&0&\ldots&\ldots&0\\
	0&(1+\xi)&0&\ldots&\ldots&0\\
	\vdots & 0&\ddots&\ldots&\ldots&0\\
	\vdots&\vdots&\vdots&\ddots&\ldots&\vdots\\
	\vdots&\vdots&\vdots&\vdots&\ddots&0\\
	0&0&\ldots&\ldots&0&(1+\xi).
\end{array}\right).$$
Hence, the eigenvalues of $\phi$ are $\lambda_1=(1+\xi)^2$, $\lambda_2=\cdots=\lambda_n=(1+\xi)$ and the corresponding eigenvectors are $\{b,u_2,\cdots,u_n\}$.\\
Now according to the Proposition \ref{29} the Levi-Civita connection of the Riemannian metric $g_X$ is as follows:
\begin{align*}
	&\nabla_{u_j}b=-u_j,\quad j\neq1,\\
	&\nabla_{u_j}u_j=\frac{1}{1+\xi}b,\quad j\neq1,\\
	&\nabla_{u_i}u_j=0, i\neq j \quad and\quad i,j=2,3,\cdots,n ,\\
	&\nabla_{b}b=0\\
	&\nabla_{b}u_j=0,\quad j=2,3,\cdots,n.
\end{align*}
Therefore, according to the Proposition \ref{31} the sectional curvature of the Riemannian metric $g_X$ is as follows:
\begin{equation*}
	K(u_i,u_j)=-\frac{1}{(1+\xi)^2}.
\end{equation*}
\begin{remark}
Finally, we mention that using formula (2.8) of \cite{Parhizkar-salimi}, the results of this paper can be generalized in an almost similar way of Randers metrics to the other $(\alpha,\beta)$-metrics.
\end{remark}

{\large{\textbf{Acknowledgment.}}} We are grateful to the office of Graduate Studies of the University of Isfahan for their support.
\section*{Declarations}

{\textbf{Ethical Approval:} Not applicable.\\

{\textbf{Conflict of interest:} On behalf of all authors, the corresponding author states that there is no conflict of interest.\\

{\textbf{Authors contributions:}} Not applicable.\\

{\textbf{Funding:}} There is not any financial support. \\

{\textbf{Availability of data and materials:}} Data sharing does not apply to this article as no datasets were generated or analyzed during the current study.\\

\end{document}